\documentclass[11pt, letterpaper,reqno]{amsart}

\usepackage{amsmath,amsthm,amsrefs,amssymb,hyperref}
\usepackage[usenames,dvipsnames]{xcolor}
\usepackage{graphicx}
\usepackage[percent]{overpic}
\usepackage{todonotes}

\numberwithin{equation}{section}

\newtheorem{corollary}{Corollary}[section]
\newtheorem{lemma}{Lemma}[section]
\newtheorem{theorem}{Theorem}[section]

\theoremstyle{definition}

\DeclareMathOperator{\D}{\mathbb{D}}
\DeclareMathOperator{\C}{\mathbb{C}}

\begin{document}
\title{The Bergman number of a plane domain}

\author{Christina Karafyllia}  
\address{Institute for Mathematical Sciences, Stony Brook University, Stony Brook, NY 11794, U.S.A.}
\email{christina.karafyllia@stonybrook.edu}
\address{Department of Mathematics, University of Thessaly, Lamia, 35100, Greece.}
\email{ckarafyllia@uth.gr}   
\thanks{I would like to thank Gregory Markowsky for the valuable comments and the referee for the corrections and especially for providing the idea for the proof of  Corollary \ref{equideff}.}          



\begin{abstract} Let $D$ be a domain in the complex plane $\C$. The Hardy number of $D$, which first introduced by Hansen, is the maximal number $h(D)$ in $[0,+\infty]$ such that $f$ belongs to the classical Hardy space $H^p (\D)$ whenever $0<p<h(D)$ and $f$ is holomorphic on the unit disk $\D$ with values in $D$. As an analogue notion to the Hardy number of a domain $D$ in $\C$, we introduce the Bergman number of $D$ and we denote it by $b(D)$. Our main result is that, if $D$ is regular, then $h(D)=b(D)$. This generalizes earlier work by the author and Karamanlis for simply connected domains. The Bergman number $b(D)$ is the maximal number in $[0,+\infty]$ such that $f$ belongs to the weighted Bergman space $A^p_{\alpha} (\D)$ whenever $p>0$ and $\alpha>-1$ satisfy $0<\frac{p}{\alpha+2}<b(D)$ and $f$ is holomorphic on $\D$ with values in $D$. We also establish several results about Hardy spaces and weighted Bergman spaces and we give a new characterization of the Hardy number and thus of the Bergman number of a regular domain with respect to the harmonic measure. 
\end{abstract}

\maketitle

\section{Introduction}\label{int}

The Hardy space with exponent $p>0$ is denoted by $H^p (\D)$ and is defined to be  the set of all holomorphic functions $f$ on the unit disk $\D$ such that
\[ \sup_{0<r<1}\int_{0}^{2\pi} {{{| {f( re^{i\theta} )}|}^p} d\theta}  <  +\infty.\]
The fact that a function belongs to $H^p (\D)$ imposes a restriction on its growth and this restriction is stronger when $p$ increases, that is, if $0<q<p$ then $H^p (\D) \subset H^q (\D)$. For the theory of Hardy spaces see \cite{Dur}.

In \cite{Han} Hansen studied the problem of determining the numbers $p>0$ for which a holomorphic function $f$ on $\mathbb{D}$ belongs to ${H^p}( \mathbb{D})$ by studying $f( \mathbb{D} )$. For this purpose, he introduced a number which he called the Hardy number of a domain. The Hardy number of a domain $D$ in the complex plane $\C$ is defined as 
\[h(D)=\sup \left\{ p>0:|z|^p \,\, \text{has a harmonic majorant on}\,\,D \right\}.\]
Later, in \cite{Kim} Kim and Sugawa proved some equivalent definitions for $h(D)$. Let $f$ be a holomorphic function on $\D$ and set
\[h(f)=\sup \{p>0:f\in H^p(\D)\}\in [0,+\infty].\]
Since $H^p(\D) \subset H^q(\D)$ for $0<q<p$, it follows that $f\in H^p(\D)$ for $p<h(f)$ and $f\notin H^p(\D)$ for $p>h(f)$. In case $p=h(f)$, then both can happen (see \cite{Karark}). Let $D$ be a domain in $\C$. The Hardy number of $D$ can equivalently be characterized (see \cite[Lemma 2.1]{Kim}) as 
\[h(D)=\inf \{h(f):f\in H(\D,D)\}\in [0,+\infty]\]
or
\[h(D)=\sup \{p>0 : H(\D,D)\subset H^p (\D)\},\]
where $H(\D,D)$ denotes the set of all holomorphic functions on $\D$ with values in $D$. Combining the definitions above with the monotonicity of Hardy spaces, we infer that for a given plane domain $D$, the Hardy number of $D$, $h(D)$, is the maximal number in $[0,+\infty]$ such that $f \in H^p (\D)$ whenever $0<p<h(D)$ and $f$ is holomorphic on $\D$ with values in $D$. Moreover, Kim and Sugawa \cite[Lemma 2.1]{Kim} proved that $h(D)=h(f)$ for a holomorphic universal covering map of $\D$ onto $D$.

The Hardy number has been studied extensively over the years. A classical problem is to find estimates or exact descriptions for it.
What we know, so far, about the Hardy number of an arbitrary plane domain is some estimates proved by Hansen in \cite{Han} and an exact formula for it involving harmonic measure proved by Kim and Sugawa in \cite{Kim}. Their proof is based on  Ess{\'e}n's main lemma in \cite{Ess}. However, we know more on how to estimate it for certain types of domains such as starlike \cite{Han} and spiral-like domains \cite{Han2}, comb domains \cite{Karfin} and, more generally, simply connected domains \cite{Karark}, \cite{Kar}. In this paper, we prove one more way to compute the Hardy number of a regular domain with the aid of harmonic measure (see Theorem \ref{hardnumb}).

A more general class of holomorphic functions than Hardy spaces is weighted Bergman spaces. The weighted Bergman space with exponent $p>0$ and weight $\alpha>-1$ is denoted by $A_\alpha ^p (\D)$ and is defined to be the set of all holomorphic functions $f$ on $\mathbb{D}$ such that
\[\int_\mathbb{D} {{{\left| {f\left( z \right)} \right|}^p}{{\left( {1 - {{\left| z \right|}^2}} \right)}^\alpha }dA\left( z \right)}  < +\infty,\] 
where $dA$ denotes the Lebesgue area measure on $\mathbb{D}$. The unweighted Bergman space ($\alpha=0$) is simply denoted by $A^p (\D)$ and it is known as the Bergman space with exponent $p$. Weighted Bergman spaces contain Hardy spaces, that is, $H^p (\D)\subset A_{\alpha}^p (\D)$, for all $\alpha>-1$ and $p>0$ (see \cite{Zhu}). Actually, a more general result holds. By Corollary 4.4 in \cite{Smith} it follows that $H^q (\D)\subset A_{\alpha}^p (\D)$ provided that $\frac{p}{\alpha+2} \le q\le p$. For the theory of Bergman spaces see \cite{DurS}.

As an analogue notion to the Hardy number of a plane domain, here we introduce the Bergman number of a domain in $\C$ in the following way. If $f$ is a holomorphic function on $\D$, we set
\[b(f)=\sup \left\{\frac{p}{\alpha+2}:p>0, \alpha>-1,f\in A^p_\alpha(\D)\right\}\in [0,+\infty].\]
Note that the number $b(f)$ was first introduced by the author and Karamanlis in \cite{Kar} for conformal maps $f$ on $\D$ and it was proved that $b(f)=h(f)$. Let $D$ be a domain in $\C$. We define the Bergman number of $D$ as
\[b(D)=\inf \{b(f):f\in H(\D,D)\}\in [0,+\infty].\]
Our first observation is that the Bergman number satisfies the same basic properties as the Hardy number (see \cite[p.\ 236]{Han} or \cite[Lemma 2.3]{Kim}). That is,
\begin{enumerate}
	\item if $D$ is bounded then $b(D)=+\infty$, 
	\item if $\C \backslash D$ is bounded then $b(D)=0$,
	\item if $D\subset D'$ then $b(D)\ge b(D')$,
	\item $b(D)=b(\phi (D))$ for a complex affine map $\phi (z)=az+b$ with $a\neq 0$,
	\item if $D$ is simply connected then $b(D)\ge 1/2$.
\end{enumerate}
Properties (1), (2) are proved in Section \ref{se1}, properties (3), (4) are trivial and (5) is proved in \cite{Kar}. Moreover, we observe that since $H^p (\D)\subset A_{\alpha}^p (\D)$, it is easy to show that $h(D)\le b(D)$ (see Section \ref{se1}). The question which arises is whether the reverse inequality holds and thus  $h(D)= b(D)$,  which is not a direct consequence of their definitions. In this paper we prove that if $D$ is regular then the answer is positive. We define a subdomain $D$ of $\C_\infty=\C \cup \{\infty\}$ to be regular if $D$ posseses a ``barrier function" at each of its boundary points \cite[Chapter 4]{Ran}. If $D$ is an unbounded subdomain of $\C$, then $\infty$ is to be included among the boundary points of $D$. 

\begin{theorem}\label{main}
Let $D$ be a regular domain and $f$ be a universal covering map of $\D$ onto $D$. Then $h(D)=b(D)=h(f)=b(f)$.
\end{theorem}

An immediate corollary (see Section \ref{lastse}) is that if $\frac{p}{\alpha+2}<b(D)$ then $f\in A^p_{\alpha} (\D)$ for every holomorphic function $f$ on $\D$ with values in $D$. In other words, $b(D)$ is the maximal number in $[0,+\infty]$ such that $f\in A^p_{\alpha} (\D)$ whenever $0<\frac{p}{\alpha+2}<b(D)$ and $f$ is holomorphic on $\D$ with values in $D$. Another consequence of Theorem \ref{main} (see Section \ref{lastse}) is that, if $D$ is regular, then
\[b(D)=\sup \left\{\frac{p}{\alpha+2}: p>0, \alpha>-1, H(\D,D)\subset A^p_{\alpha} (\D)\right\}.\]

We also establish the following result that gives a new description of the Hardy number and thus of the Bergman number of a regular domain involving harmonic measure. For the definition of harmonic measure see \cite{Gar}, \cite{Ran}.

\begin{theorem}\label{hardnumb} Let $D$ be a regular domain and let $a\in D$. If $E_r=\partial D \cap \{z \in \C: |z| > r\}$ for $r>0$ (see Fig.\ref{figg}), then
	\[h(D)=b(D)=\liminf_{r\to +\infty}\frac{\log\omega_{D}(a,E_r)^{-1}}{\log r}.\]
\end{theorem}

We note that Theorem \ref{hardnumb} is not true if $D$ is not regular. For example, if $D=\{z\in \C:|z|>1\}$, then $h(D)=b(D)=0$ because $\C \backslash D$ is bounded. However, $\omega_{D}(a,E_r)=0$ for every $r>1$.

Next, we prove the results above as follows. In Section \ref{se1} we show that if $D$ is a plane domain, then $h(D)\le b(D)$. For the proof of the reverse inequality, which is more complicated and occupies the rest of the paper, in Section \ref{se2} we establish some results about covering maps in Hardy and weighted Bergman spaces, which are of independent interest. Then, in Section \ref{lastse} we prove Theorem \ref{main},  Theorem \ref{hardnumb} and some subsequent results.

	\begin{figure} 
	\begin{center}
		\includegraphics[scale=0.5]{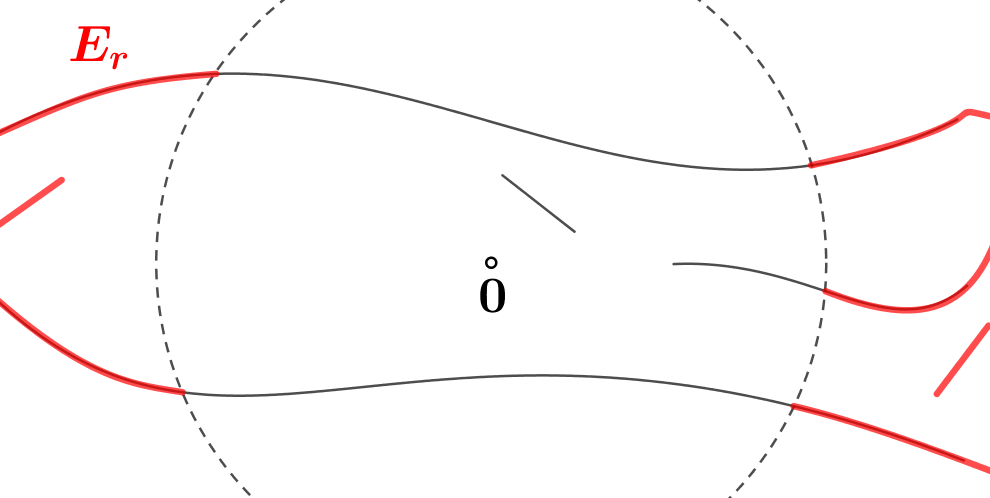}
		\caption{The set $E_r$.}
		\label{figg}
	\end{center}
\end{figure}

\section{Proof of $h(D)\le b(D)$}\label{se1}

\begin{lemma}\label{lemma11}
 Let $f$ be a holomorphic function on $\D$. Then $h(f)\le b(f)$.
\end{lemma}

\begin{proof} Let $f$ be a holomorphic function on $\D$. If $h(f)=0$ then, obviously, $h(f)\le b(f)$. If $h(f)=+\infty$ then $f \in H^p(\D)$ for every $p>0$. Since $H^p(\D)\subset A^p_\alpha(\D)$, it follows that $f \in A^p_\alpha(\D)$ for every $p>0$ and for every $\alpha>-1$. Therefore, we have 
	\[b(f)\ge \frac{p}{\alpha+2}\]
	for every $p>0$ and for every $\alpha>-1$. Setting $\alpha=0$ and then letting $p\to +\infty$, we deduce that $b(f)=+\infty$. Now, suppose that $h(f)=p_0>0$. Then $f\in H^p(\D)$ for every $0<p<p_0$. This implies that $f \in A^{2p} (\D)$ for every $0<p<p_0$ (see \cite{Smith}) and thus 
\[b(f) \ge \frac{2p}{0+2} = p.\]  
This proves $b(f) \ge p_0$ or $b(f) \ge h(f)$. Therefore, in any case we have that $h(f)\le b(f)$. This holds for any holomorphic function on $\D$ and the proof is complete.
	\end{proof}

\begin{lemma}\label{lemma12}
Let $D$ be a domain in $\C$. Then $h(D)\le b(D)$. Moreover, if $D$ is bounded, then $h(D)=b(D)=+\infty$ and if $\C \backslash D$ is bounded, then $h(D)=b(D)=0$.
\end{lemma}

\begin{proof}
Since, by Lemma \ref{lemma11}, $h(f)\le b(f)$ for every $f\in H(\D,D)$, by the definitions of $h(D)$ and $b(D)$ we directly have $h(D)\le b(D)$. 

If $D$ is bounded, then $h(D)=+\infty$ (see \cite[p.\ 236]{Han} or \cite[p.\ 293]{Kim}). Since $h(D)\le b(D)$, it follows that $b(D)=+\infty$.

If $\C \backslash D$ is bounded, then $h(D)=0$ (see \cite[p.\ 236]{Han} or \cite[Lemma 2.3]{Kim}). The proof that $b(D)=0$ is similar. Since $\C \backslash D$ is bounded, $\C \backslash D \subset \{z\in \C: |z|\le R \}$ for some $R>0$. So, the function 
\[f(z)=Re^{\frac{1+z}{1-z}}\]
maps $\D$ into $D$. For $\varepsilon>0$ sufficiently close to $0$ we have the following estimates
\begin{align}
\int_\mathbb{D} {{{\left| {f\left( z \right)} \right|}^p}{{\left( {1 - {{\left| z \right|}^2}} \right)}^\alpha }dA\left( z \right)}&=R^p\int_0^1 \int_0^{2\pi} e^{\frac{(1-r^2)p}{1+r^2-2r\cos \theta}}(1-r^2)^{\alpha}rd\theta dr \nonumber\\
&\ge R^p\int_0^1 \int_0^{1-r} e^{\frac{(1-r^2)p}{1+r^2-2r\cos \theta}}(1-r^2)^{\alpha}rd\theta dr \nonumber\\
&\ge R^p\int_0^1 (1-r)e^{\frac{(1-r^2)p}{1+r^2-2r\cos (1-r)}}(1-r^2)^{\alpha}rdr \nonumber\\
&\ge R^p\int_0^1 e^{\frac{(1-r)p}{1+r^2-2r\cos (1-r)}}(1-r)^{\alpha+1}rdr \nonumber\\
&\ge R^p\int_{1-\varepsilon}^1 e^{\frac{C}{ 1-r}}(1-r)^{\alpha+1}rdr \nonumber \\  
&\ge (1-\varepsilon) R^pC'\int_{1-\varepsilon}^1 \frac{1}{(1-r)^2}dr=+\infty, \nonumber 
\end{align}
where $C,C'$ are positive constants. Therefore, $f$ does not belong to $A^p_{\alpha} (\D)$ for any $p>0$ and $\alpha>-1$. This implies that $b(f)=0$ and thus $b(D)=0$.
\end{proof}

\section{Auxiliary results for Hardy and Bergman spaces}\label{se2}

From now on, let $D$ be a regular domain. In this section we prove some results for Hardy and weighted Bergman spaces that we use later to show $h(D)\ge b(D)$. By Lemma \ref{lemma12} it suffices to assume that $D$ is unbounded. Also, we suppose $f$ is a universal covering map of $\D$ onto $D$ and we set $E_r=\partial D \cap \{z \in \C: |z| > r\}$ for $r>0$. 

\subsection{Results for Hardy spaces}\label{subse} Next, we establish a new characterization of the Hardy number of $D$. For this purpose we derive several results for covering maps in Hardy spaces. Note that the proofs below follow methods applied in \cite{Kar}.

\begin{theorem}\label{hardyequi} Let $D$ and $f$ be as above. Then $f\in H^p (\D)$ if and only if 
\[\int_{|f(0)|}^{+\infty } {t^{p-1} \omega_D (f(0),E_t)}dt<+\infty.\]	
\end{theorem}

Note that Theorem \ref{hardyequi} is not true if $D$ is not regular. For example, if $D=\{z\in \C:|z|>1\}$, then  $f\notin H^p (\D)$ for any $p>0$ but $\omega_D (f(0),E_t)=0$ for every $t>1$. Also, note that the proof below is based on Lemma 2 in \cite{Bae} that requires the regularity of $D$.

\begin{proof} It is known (see \cite{Yam}) that $f\in H^p(\D)$ if and only if 
\begin{equation}\label{haste}
\int_{\D} |f(z)|^{p - 2} | f '(z)|^2   \log \frac{1}{|z|}dA(z)<+\infty.
\end{equation}	
Since $D$ is regular, it possesses a Green function (for the definition of the Green function see \cite{Gar}). For the Green function of $D$, we set $g_D(f(0),w)=0$, for $w\notin D$. If $z_j(w,f)$ are the pre-images of $w$ under $f$, by the non-univalent change of variable (see \cite{Bets} or \cite[p.\ 243]{Fed}) and Lindel\"{o}f's principle (see \cite{Bets}, \cite[p.\ 235]{Mars} or \cite[p. 55, 210]{Nev}), we have
	
\begin{align}\label{sx1}
\int_{\D} |f(z)|^{p - 2} | f '(z)|^2   \log \frac{1}{|z|}dA(z)&= \int_D {|w|^{p - 2} \sum_j \log \frac{1}{|z_j(w,f)|} dA(w)} \nonumber \\
	&=\int_D {|w|^{p - 2} g_{D}(f(0),w)dA(w)}   \nonumber \\
	&=\int_0^{+\infty } {r ^{p - 1}\left( {\int_0^{2\pi } {{g_D}(f(0),re^{i\theta })d\theta } } \right)dr }.
	\end{align}
Next, we state a known relation between harmonic measure and the Green function  (see \cite[Lemma 2]{Bae}). If $\Omega$ is a regular domain and $a\in \Omega$ then, for $r>|a|$, 
	\begin{equation}
	\int_r^{+\infty}\omega_{\Omega}\left({a,\partial\Omega\cap\{|z|>t\}}\right)\frac{dt}{t}=\frac{1}{2\pi}\int_0^{2\pi} g_{\Omega}({a,re^{i\theta})}d\theta-g_{\Omega}(a, \infty). \nonumber  
	\end{equation}
Since $D$ is regular, by Theorem 4.4.9 in \cite{Ran}  we have $g_D(f(0),\infty)=0$. Thus, for every $r>|f(0)|$,
	\begin{align}\nonumber
	\int_0^{2\pi} g_D(f(0),re^{i\theta})d\theta =2\pi\int_r^{+\infty} \omega_D (f(0),E_t)\frac{dt}{t}.
	\end{align}
Integrating both sides with respect to $r$ and applying Fubini's theorem, it follows that
\begin{align}\label{sx2}
\int_{|f(0)|}^{+\infty } r ^{p - 1}&\left( {\int_0^{2\pi } {{g_D}(f(0),re^{i\theta })d\theta } } \right)dr \nonumber\\ 
&=2\pi\int_{|f(0)|}^{+\infty } r^{p-1} \left( \int_r^{+\infty} \omega_D (f(0),E_t)\frac{dt}{t} \right)dr \nonumber \\
&=2\pi \int_{|f(0)|}^{+\infty } \frac{\omega_D (f(0),E_t)}{t} \left( \int_{|f(0)|}^t r^{p-1}dr \right)dt \nonumber \\
&= \frac{2\pi}{p} \int_{|f(0)|}^{+\infty } {\omega_D (f(0),E_t)} t^{p-1}\left(1-\frac{|f(0)|^p}{t^p} \right)dt.
\end{align}	
Now, suppose that $f\in H^p(\D)$. By (\ref{sx2}) we have
\begin{align}
\int_{|f(0)|}^{+\infty } r ^{p - 1}&\left( {\int_0^{2\pi } {{g_D}(f(0),re^{i\theta })d\theta } } \right)dr  \nonumber\\
&\ge \frac{2\pi}{p} \int_{2|f(0)|}^{+\infty } {\omega_D (f(0),E_t)} t^{p-1}\left(1-\frac{|f(0)|^p}{t^p} \right)dt \nonumber \\
&\ge \frac{2\pi}{p} \left(1-\frac{1}{2^p}\right) \int_{2|f(0)|}^{+\infty } {\omega_D (f(0),E_t)} t^{p-1}dt. \nonumber
\end{align}
This in conjunction with (\ref{haste}) and (\ref{sx1}) implies that
\[\int_{|f(0)|}^{+\infty } {t^{p-1} \omega_D (f(0),E_t)}dt<+\infty.\]
Conversely, suppose that 
\[\int_{|f(0)|}^{+\infty } {t^{p-1} \omega_D (f(0),E_t)}dt<+\infty.\]
By (\ref{sx2}) we have
\begin{equation}
\int_{|f(0)|}^{+\infty } {r ^{p - 1}\left( {\int_0^{2\pi } {{g_D}(f(0),re^{i\theta })d\theta } } \right)dr } \le \frac{2\pi}{p} \int_{|f(0)|}^{+\infty } {t^{p-1} \omega_D (f(0),E_t)}dt<+\infty. \nonumber
\end{equation}
This in combination with (\ref{haste}) and (\ref{sx1}) implies that $f\in H^p(\D)$. 
\end{proof}

\begin{corollary}\label{lemma21} Let $D$ and $f$ be as above. If $f\in H^p (\D)$, then there exist constants $C>0$ and $r_0>0$ such that, for every $r>r_0$, 
\[\omega_{D}(f(0),E_r)\le Cr^{-p}.\]
If there are constants $C>0$, $p'>0$ and $r_0>0$ such that, for every $r>r_0$, 		
\[\omega_{D}(f(0),E_r)\le Cr^{-p'},\]
then $f\in H^p (\D)$ for every $0<p<p'$.
\end{corollary}

\begin{proof} Let $f\in H^p(\D)$. By Theorem \ref{hardyequi} we have  
	\[\int_{|f(0)|}^{+\infty}r^{p-1}\omega_{D}(f(0),E_r)dr<+\infty.
	\]
	Since $\omega_{D}(f(0),E_r)$ is decreasing in $r$, it follows that, for $R>2|f(0)|$,
	\begin{align}
	+\infty>\int_{|f(0)|}^{+\infty}r^{p-1}\omega_{D}(f(0),E_r)dr & \ge \int_{|f(0)|}^{R}r^{p-1}\omega_{D}(f(0),E_r)dr \nonumber \\
	&\ge \omega_{D}(f(0),E_R)\int_{|f(0)|}^{R}r^{p-1}dr \nonumber \\
&=\frac{R^p}{p}\left( 1-\frac{|f(0)|^p}{R^p} \right)\omega_{D}(f(0),E_R) \nonumber \\
&\ge \frac{R^p}{p}\left( 1-\frac{1}{2^p} \right)\omega_{D}(f(0),E_R). \nonumber
\end{align}
Therefore, there is a constant $C>0$ such that  for every $R>2|f(0)|$,
	\[\omega_{D}(f(0),E_R)\le CR^{-p}.\]
	
Conversely, suppose that there are constants $p'>0$, $C>0$ and $r_0>0$ such that 
	\[\omega_{D}(f(0),E_r)\le Cr^{-p'}\]
	for every $r>r_0$. If $0<p<p'$, then we have
	\[\int_{r_0}^{+\infty}r^{p-1}\omega_{D}(f(0),E_r)dr\le C\int_{r_0}^{+\infty}r^{p-1-p'}dr<+\infty.\]
	Thus, by Theorem \ref{hardyequi} it follows that $f\in H^p(\D)$ for every $0<p<p'$.
\end{proof}

Now, we are able to prove a new characterization for the Hardy number of a regular domain.

\begin{lemma}\label{hanu1} Let $D$ be a regular domain and $a\in D$. If $E_r=\partial D \cap \{z \in \C: |z| > r\}$ for $r>0$, then
	\[h(D)=\liminf_{r\to +\infty}\frac{\log\omega_{D}(a,E_r)^{-1}}{\log r}.\]	
\end{lemma}
	
\begin{proof} Let $f$ be a universal covering map of $\D$ onto $D$ such that $f(0)=a$. If $f\in H^p(\D)$, then by Corollary \ref{lemma21} there are constants $C>0$ and $r_0>0$ such that
	\[\omega_{D}(a,E_r)\le Cr^{-p},\]
	for every $r>r_0$. Consequently, for $r>\max \left\{1,r_0\right\}$,
	\[\frac{\log \omega_{D}(a,E_r)^{-1}}{\log r}\ge \frac{\log C^{-1}}{\log r}+ p.\]
Taking limits as $r\to +\infty$, we deduce that
	\[\liminf_{r\to +\infty}\frac{\log\omega_{D}(a,E_r)^{-1}}{\log r}\ge p.\]
	This holds for any $p>0$ for which $f\in H^p(\D)$ and hence
	\begin{align}\label{miafora}
	\liminf_{r\to +\infty}\frac{\log\omega_{D}(a,E_r)^{-1}}{\log r}\ge h(f).
	\end{align}
Now, we set 
	\[I:=\liminf_{r\to +\infty}\frac{\log\omega_{D}(a,E_r)^{-1}}{\log r}.\]
	If $p<I$, then there exist $\varepsilon>0$ and $r_0>0$ such that, for every $r>r_0$,
	\[
	p+\varepsilon\le\frac{\log\omega_{D}(a,E_r)^{-1}}{\log r}
	\]
	or, equivalently,
	\[\omega_{D}(a,E_r)\le r^{-p-\varepsilon}.\]
	By Corollary \ref{lemma21} we deduce that $f\in H^p(\D)$. This shows that the  interval $(0,I)$ is contained in the set
	\[
	\left\{p>0:\ f\in H^p(\D)\right\}
	\]	
and hence $h(f)\ge I $ or
	\begin{align}\nonumber
	\liminf_{r\to +\infty}\frac{\log\omega_{D}(a,E_r)^{-1}}{\log r}\le h(f).
	\end{align}
	This in conjunction with (\ref{miafora}) gives 
\[h(f)=\liminf_{r\to +\infty}\frac{\log\omega_{D}(a,E_r)^{-1}}{\log r}.\]
Since $f$ is a universal covering map of $\D$ onto $D$, we have that $h(D)=h(f)$ \cite[Lemma 2.1]{Kim} and thus we derive that 
\[h(D)=\liminf_{r\to +\infty}\frac{\log\omega_{D}(a,E_r)^{-1}}{\log r},\]
which completes the proof.
\end{proof}

\subsection{Results for weighted Bergman spaces} 
Next, we prove some results for covering maps in weighted Bergman spaces. Recall that we suppose $D$ is an unbounded regular domain. Also, $f$ is a universal covering map of $\D$ onto $D$ and $E_r=\partial D \cap \{z \in \C: |z| > r\}$ for $r>0$. The following proofs are similar to the proofs of Section \ref{subse}.

\begin{theorem}\label{bergmansp} Let $D$ and $f$ be as above. If $f\in A^p_{\alpha} (\D)$ then 
	\[\int_{|f(0)|}^{+\infty} r^{p-1} \omega_{D}(f(0),E_r)^{\alpha+2}dr<+\infty.\] 
\end{theorem}

\begin{proof} Since $f\in A_\alpha^p(\D)$, it follows (see \cite[p.\ 2336]{Smith}) that
	\begin{equation}\label{finite}
	\int_{\D} {|f(z)|^{p - 2} | f '(z)|^2 \left(\log \frac{1}{|z|}\right)^{\alpha+2}dA(z)}<+\infty.
	\end{equation}
	Since $D$ is regular, it possesses a Green function. For the Green function of $D$, we set $g_D(f(0),w)=0$, for $w\notin D$. If $z_j(w,f)$ are the pre-images of $w$ under $f$, by the non-univalent change of variable (\cite[p.\ 243]{Fed}) and Lindel\"{o}f's principle (see \cite{Bets}, \cite[p.\ 235]{Mars}), we have
	
	\begin{align}\label{equal}
	\int_{\D} |f(z)|^{p - 2}& | f '(z)|^2  \left( \log \frac{1}{|z|}\right)^{\alpha+2}dA(z) \nonumber \\
	&=\int_D {|w|^{p - 2} \left( \sum_j \log \frac{1}{|z_j(w,f)|} \right) ^{\alpha+2}dA(w)} \nonumber \\
	&= \int_D {|w|^{p - 2} g_{D}(f(0),w)^{\alpha+2}dA(w)}   \nonumber \\
	&=\int_0^{+\infty } {r ^{p - 1}\left( {\int_0^{2\pi } {{g_D}( f(0),re^{i\theta })^{\alpha+2}d\theta } } \right)dr }.
	\end{align}
	Since $\alpha+2>1$, by Jensen's inequality we derive that, for every $r>0$,
	\begin{equation}
	\left( \frac{1}{2\pi} \int_0^{2\pi} {g_D(f(0),re^{i\theta})d\theta} \right)^{\alpha+2} \le\frac{1}{2\pi} \int_0^{2\pi} {{g_D}(f(0),re^{i\theta })^{\alpha+2}d\theta}. \nonumber
	\end{equation}
	This in conjunction with (\ref{finite}) and (\ref{equal}) implies that
	\begin{equation}\label{doubleint}
	\int_0^{ +\infty } {r ^{p - 1}\left( {\int_0^{2\pi } {{g_D}( f(0),re^{i\theta })d\theta } } \right)^{\alpha+2}dr }<+\infty.
	\end{equation}
By Lemma 2 in \cite{Bae} we have that  for every $r>|f(0)|$,
	\begin{align}\label{baern}
	\int_0^{2\pi} g_D(f(0),re^{i\theta})d\theta &=2\pi\int_r^{+\infty} \omega_D (f(0),E_t)\frac{dt}{t} \nonumber\\
	&\ge 2\pi \int_r^{2r} \omega_{D}(f(0),E_t)\frac{dt}{t} \nonumber \\
	& \ge 2\pi \log 2\, \omega_{D}(f(0),E_{2r}).
	\end{align}
	The last estimate comes from the fact that $\omega_{D}(0,E_t)$ is decreasing in $t>0$. Therefore, by (\ref{baern}) and (\ref{doubleint}) we have
\begin{align}
\int_{|f(0)|}^{+\infty} r^{p-1} \omega_{D}(f(0),E_{2r})&^{\alpha+2}dr<+\infty. \nonumber 
\end{align}
Hence, by a change of variable  we obtain
	\[\int_{2|f(0)|}^{+\infty} r^{p-1} \omega_{D}(f(0),E_r)^{\alpha+2}dr<+\infty,\]
which implies the desired result.
\end{proof} 

\begin{corollary}\label{ber}
Let $D$ and $f$ be as above. If $f\in A^p_{\alpha} (\D)$ then there are constants $C>0$ and $r_0>0$ such that,
for every $r>r_0$,
\[\omega_{D}(f(0),E_r)\le Cr^{-\frac{p}{\alpha+2}}.\]
\end{corollary}

\begin{proof} Since $f\in A_{\alpha}^p(\D)$, by Theorem \ref{bergmansp} we have  
\[\int_{|f(0)|}^{+\infty}r^{p-1}\omega_{D}(f(0),E_r)^{\alpha+2}dr<+\infty.
\]
Since $\omega_{D}(f(0),E_r)$ is decreasing in $r$, it follows that for $R>2|f(0)|$,
\begin{align}
+\infty>\int_{|f(0)|}^{+\infty}r^{p-1}\omega_{D}(f(0),E_r)^{\alpha+2}dr & \geq\int_{|f(0)|}^{R}r^{p-1}\omega_{D}(f(0),E_r)^{\alpha+2}dr \nonumber\\
& \ge \omega_{D}(f(0),E_R)^{\alpha+2}\int_{|f(0)|}^{R}r^{p-1}dr \nonumber \\
&\ge \frac{R^p}{p}\left( 1- \frac{1}{2^p}\right)\omega_{D}(f(0),E_R)^{\alpha+2}. \nonumber
\end{align}
Therefore, there is a constant $C>0$ such that, for every $R>2|f(0)|$,
\[\omega_{D}(f(0),E_R)\le CR^{-\frac{p}{\alpha+2}}\]
and the proof is complete.
\end{proof}

\section{Proof of the main results}\label{lastse}

Now, we are able to prove Theorem \ref{main} and Theorem \ref{hardnumb}.

\begin{proof}[Proof of Theorem \ref{main}] 
Set $f(0)=a$. If $f\in A_\alpha^p(\D)$, then by Corollary \ref{ber} there are constants $C>0$ and $r_0>0$ such that
\[\omega_{D}(a,E_r)\le Cr^{-\frac{p}{\alpha+2}},\]
for every $r>r_0$. Consequently, for $r>\max \left\{1,r_0 \right\}$,
\[\frac{\log \omega_{D}(a,E_r)^{-1}}{\log r}\ge \frac{\log C^{-1}}{\log r}+ \frac{p}{\alpha+2}.\]
Taking limits as $r\to +\infty$, we have
\[\liminf_{r\to +\infty}\frac{\log\omega_{D}(a,E_r)^{-1}}{\log r}\ge\frac{p}{\alpha+2}.\]
This holds for any $p>0$ and $\alpha>-1$ for which $f\in A_\alpha^p(\D)$ and hence
\begin{align}\label{last}
\liminf_{r\to +\infty}\frac{\log\omega_{D}(a,E_r)^{-1}}{\log r}\ge b(f) \ge b(D).
\end{align}	
The last inequality comes from the definition of $b(D)$. Since $f$ is a universal covering map of $\D$ onto $D$, we have that $h(D)=h(f)$ \cite[Lemma 2.1]{Kim}. By  Lemma \ref{hanu1} and  (\ref{last}) it follows that 
\[h(D)=h(f)\ge b(f)\ge b(D).\]
Moreover, by Lemma \ref{lemma12} we have $h(D)\le b(D)$ and thus 
\[h(D)=h(f)= b(f)= b(D),\]
which completes the proof.
\end{proof}

\begin{proof} [Proof of Theorem \ref{hardnumb}]
It follows from Lemma \ref{hanu1} and Theorem \ref{main}.
\end{proof}

Next, we prove some consequences of Theorem \ref{main}.

\begin{corollary}\label{equideff} Let $D$ be a regular domain. Then
\[b(D)=\sup \left\{\frac{p}{\alpha+2}: p>0, \alpha>-1, H(\D,D)\subset A^p_{\alpha} (\D)\right\}.\] 
\end{corollary}

\begin{proof} We set
\[b'(D)=\sup \left\{\frac{p}{\alpha+2}: p>0, \alpha>-1, H(\D,D)\subset A^p_{\alpha} (\D)\right\}.\]
Fix $q>0$ such that $H(\D,D)\subset H^q(\D)$. Now, we state a known result. If we set $\phi (z)=z$ and $\beta=-1$ in Corollary 4.4 in \cite{Smith}, it follows that $H^{q}(\D)\subset A_{\alpha}^p(\D)$ provided that $\frac{p}{\alpha+2} \le q \le p$. Therefore, $H^{q}(\D)\subset A_{\alpha}^p(\D)$ for every $p>0$ and $\alpha>-1$ such that $\frac{p}{\alpha+2} = q$ (this also follows from Theorem 5.11 in \cite{Dur}). Since $H(\D,D)\subset H^q(\D)$, we infer that $H(\D,D)\subset A_{\alpha}^p(\D)$ for every $p>0$ and $\alpha>-1$ such that $\frac{p}{\alpha+2} = q$. This implies that $b'(D)\ge q$. So, $b'(D)\ge q$ for any $q>0$ satisfying $H(\D,D)\subset H^q(\D)$  and thus, by the definition of $h(D)$, we have  $b'(D)\ge h(D)$.

Conversely, let $f$ be a universal covering map of $\D$ onto $D$. Corollary \ref{lemma21} and Corollary \ref{ber} imply that if $f \in A^p_{\alpha}(\D)$ for some $p>0$ and $\alpha>-1$, then $f\in H^q(\D)$ for any $0<q<\frac{p}{\alpha+2}$. Fix $p>0$ and $\alpha>-1$ such that $H(\D,D)\subset A^p_{\alpha}(\D)$. Since $H(\D,D)\subset A^p_{\alpha}(\D)$, then $f\in A^p_{\alpha}(\D)$ and thus $f\in H^q(\D)$ for any $0<q<\frac{p}{\alpha+2}$. So, $h(D)=h(f) \ge \frac{p}{\alpha+2}$. Since $h(D) \ge \frac{p}{\alpha+2}$ for every $p>0$ and $\alpha>-1$ satisfying $H(\D,D)\subset A^p_{\alpha}(\D)$, we infer that $h(D)\ge b'(D)$.

Therefore, $h(D)= b'(D)$ and Theorem \ref{main} implies that $b(D)= b'(D)$.
\end{proof}

We note that Corollary \ref{equideff} does not imply that if $\frac{p}{\alpha+2}<b(D)$ then $f\in A^p_{\alpha} (\D)$ whenever $f$ is holomorphic on $\D$ with values in $D$. To prove this we need Theorem \ref{main}.

\begin{corollary}\label{coro111} Let $D$ be a regular domain. Then $b(D)$ is the maximal number in $[0,+\infty]$ such that $f\in A^p_{\alpha} (\D)$ whenever $0<\frac{p}{\alpha+2}<b(D)$ and $f$ is holomorphic on $\D$ with values in $D$.
\end{corollary}

\begin{proof} Let $0<\frac{p}{\alpha+2}<b(D)$. Since, by Theorem \ref{main}, $h(D)=b(D)$, we have $0<\frac{p}{\alpha+2}<h(D)$ which implies that $f\in H^{\frac{p}{\alpha+2}} (\D)$ for any holomorphic function $f$ on $\D$ with values in $D$.  As we stated before, it is known that $H^{q}(\D)\subset A_{\alpha}^p(\D)$ provided that $\frac{p}{\alpha+2} \le q \le p$. Hence, for $q= \frac{p}{\alpha+2}$ we infer that $f\in A_{\alpha}^p (\D)$ for any holomorphic function $f$ on $\D$ with values in $D$. This completes the proof.
\end{proof}

\end{document}